\documentclass[11pt,a4paper]{amsart}
\usepackage{geometry}
\usepackage[colorlinks, linkcolor=red, anchorcolor=blue, citecolor=blue]{hyperref}
\usepackage{color}
\usepackage{amsthm,bm} %bm make \mathbf work%
\usepackage{enumerate}
\usepackage{amsfonts}
\usepackage{amsmath}
\usepackage{mathrsfs}
\usepackage{amssymb}
\usepackage{amsbsy}
\usepackage{ifthen}
\usepackage[all]{xy}
\usepackage{extarrows}
\usepackage{indentfirst}        %éŠè¡çŒ©è¿å®å
\usepackage{latexsym}        % å€çæ°å­Šå

\usepackage{graphicx}
\usepackage{cases}
\usepackage{pifont}
\usepackage{txfonts}
\usepackage{xcolor}
\usepackage{multirow}
\usepackage{caption, subcaption}
\usepackage{tikz}

\newcounter{maint}

\usepackage{listings}
\numberwithin{equation}{section}

\allowdisplaybreaks[4] %% allow to display a terrible long formulae which can't be put in one paper.

\begin{document}

\newtheorem{theorem}{Theorem}[section]

\newtheorem{lemma}[theorem]{Lemma}

\newtheorem{corollary}[theorem]{Corollary}
\newtheorem{proposition}[theorem]{Proposition}

\theoremstyle{remark}
\newtheorem{remark}[theorem]{Remark}

\theoremstyle{definition}
\newtheorem{definition}[theorem]{Definition}

\theoremstyle{definition}
\newtheorem{conjecture}[theorem]{Conjecture}

\newtheorem{example}[theorem]{Example}
\newtheorem{problem}[theorem]{Problem}

%%%%%%%%%%%%%%%%%%%%%%%%%%%%%%%%%%%%%%%%%%%%%%%
%%%%%%%%%%%%%%%%%%%%%%%%%%%%%%%%%%%%%%%%%%%%%%%%%

\def\k{\Bbbk}
\def\I{\mathbb{I}}
\def\ufo{\mathfrak{ufo}}
\newcommand{\Dchaintwo}[4]{
\rule[-3\unitlength]{0pt}{8\unitlength}
\begin{picture}(14,5)(0,3)
\put(2,4){\ifthenelse{\equal{#1}{l}}{\circle*{4}}{\circle{4}}}
\put(4,4){\line(2,0){20}}
\put(26,4){\ifthenelse{\equal{#1}{r}}{\circle*{4}}{\circle{4}}}
\put(2,10){\makebox[0pt]{\scriptsize #2}}
\put(14,8){\makebox[0pt]{\scriptsize #3}}
\put(26,10){\makebox[0pt]{\scriptsize #4}}
\end{picture}}
\title[Automorphism group of  Suzuki's Hopf algebra]
{Automorphism group of   Suzuki's  Hopf algebra}

%%irreducible
%\author[Shi]{Yuxing Shi }
%\address{School of Mathematics and Information Science, Guangzhou University,  Guangzhou 510006, P.R. China}\email{blueponder@foxmail.com}
\author[Y. Shi, N. Hu]{Yuxing Shi and Naihong Hu$^\star$}

\address{School of Mathematics and Statistics, Jiangxi Normal University, Nanchang, Jiangxi 330022, People's Republic of China}
\email{yxshi@jxnu.edu.cn}
\address{School of Mathematical Sciences, MOE Key Laboratory of Mathematics and Engineering Applications \& Shanghai Key Laboratory of PMMP, East China Normal University, Shanghai 200241, China}
\email{nhhu@math.ecnu.edu.cn}
\thanks{$^\star$ Corresponding author.}

%    General info
%\noindent 2010 \emph{Mathematics Subject Classification:}\,  17B37, 20G42, 16W20.  \newline
%\subjclass[2010]{Primary 17B37, 81R50; Secondary 17B35}
\makeatletter
\@namedef{subjclassname@2020}{\textup{2020} Mathematics Subject Classification}
\makeatother
%\date{October 22, 2013}
\thanks{
%2010 Mathematics Subject Classification: 16T05.\\
\textit{Keywords:} Automorphism group; Hopf algebra;  Yetter-Drinfeld module.
%\\
%This work was partially supported by
%Foundation of Jiangxi Educational Committee (GJJ191681).
}

\begin{abstract}
In this paper, we calculate explicitly  automorphism group of  the Suzuki's Hopf algebra
$A_{Nn}^{\mu\lambda}$ by viewing Yetter-Drinfeld modules
as invariants of  Hopf algebra automorphisms.
\end{abstract}
\maketitle
%\tableofcontents

\section{Introduction}
In general, automorphism groups of algebras are difficult to determine. In fact, some algebras contain wild automorphisms, see \cite{Joseph1976, Shestakov2004}.
This leads to a question: what invariants of an algebra control its automorphism group?
Yakimov  \cite{Yakimov2014} proved the Andruskiewitsch-Dumas conjecture \cite{Andruskiewitsch2008} which concerns
the automorphism groups of the quantum nilpotent algebras $\mathcal{U}_q^+(\mathfrak{g})$
for all simple Lie algebras $\mathfrak{g}$. His proof exhibits a general classification method for
automorphism groups of related algebras such as quantum cluster algebras, algebras defined by iterated
Ore extensions \cite{Goodearl2000}, and so on. The key strategy is that
$\mathrm{Aut}\left(\mathcal{U}_q^+(\mathfrak{g})\right)$ is controlled by the group of certain continuous bi-finite automorphisms of completed quantum tori. Ceken et al.
developed the discriminant method to determine
 automorphism groups of certain noncommutative algebras such as quantum Weyl groups and so on, see \cite{Ceken2015, Ceken2016, Chan2018}.  Agore and Militaru consider the invariant that controls the automorphism group of a given
(finite dimensional) algebra $A$ (associative, Lie, Poisson, etc) is the universal bialgebra
that coacts on $A$, see \cite{Agore2024, Agore2020, Militaru2022}.
Other related works, please refer to \cite{Li2017, Chen2013}.

 As for the Hopf algebra automorphisms of pointed Hopf algebras,
 Panaite et al. \cite[Lemma 1]{Panaite1999}
 and  Shilin Yang \cite{Yang2007} made use of the fact that the spaces of group-like and skew primitive
 elements are invariants of Hopf algebra automorphisms. Musson determined the Hopf automorphism group of  the the quantized enveloping algebras $U_q(\mathfrak{g})$ defined over $\Bbb{Q}(q)$
 with $q$ transcendental and the Lie algebra $\mathfrak{g}$ semisimple, according to
 its coradical filtration \cite{Chin1996}. Similar results, please refer to
 \cite{Braverman1994, Twietmeyer1992}.

Radford proved that the group ${\rm Aut}_{\rm Hopf}(H)$ of Hopf algebra $H$ is finite, if
$H$ is a semisimple Hopf algebra over a field of characteristic $0$, or that $H$ is semisimple cosemisimple
involutory Hopf algebra over a field of characteristic $p>\dim H$ \cite{Radford1990}.
In this paper, we view  Yetter-Drinfeld modules as invariants of  Hopf algebra automorphisms.
Let $H$ be a Hopf algebra.
Given any finite-dimensional  Yetter-Drinfeld module
over $H$ and any  Hopf algebra automorphisms $\psi$ of $H$,
we can build a Yetter-Drinfeld module $V^{\psi}$.   We  have $\dim V=\dim V^{\psi}$
and ${\rm Supp} (V^\psi)=\psi^{-1}({\rm Supp}(V))$, see Lemma \ref{KeyLemma}.
So it could be  efficient
to calculate $\mathrm{Aut}_{\mathrm{Hopf}}(H)$, provided a classification of simple Yetter-Drinfeld modules over $H$.

The paper is organized as follows.
%In Section 1, we introduce some background of the paper.
In Section 2, we present necessary knowledges about  Suzuki's Hopf algebra.
In Section 3, all Hopf algebra automorphisms of  Suzuki's Hopf algebra are calculated.

\section{The Suzuki Hopf algebras $A_{Nn}^{\mu\lambda}$}
Let $\k$ be  an algebraicaly  closed field of characteristic $0$.  Since our results in the paper rely on
the classification of simple Yetter-Drinfeld modules over
 Suzuki's Hopf algebra \cite{Shi2020even, Shi2020odd}, which are based on  an algebraicaly  closed field of characteristic $0$.

Suzuki introduced a family of cosemisimple Hopf algebras $A_{Nn}^{\mu\lambda}$ parametrized by integers $N\geq 1$, $n\geq 2$ and $\mu$, $\lambda=\pm 1$, and investigated various properties and structures of them \cite{Suzuki1998}.
Wakui studied  Suzuki's  Hopf algebra  from different perspectives
\cite{Wakui2010a, Wakui2019, Wakui2003}.
The first author studied the Nichols algebras of simple Yetter-Drinfeld modules over
$A_{Nn}^{\mu\lambda}$, see \cite{Shi2020even, Shi2020odd, Shi2024Suzuki}.
Suzuki's Hopf algebra  $A_{Nn}^{\mu\lambda}$ is generated by $x_{11}$, $x_{12}$, $x_{21}$,
$x_{22}$ subject to the relations:
\begin{align*}
&x_{11}^2=x_{22}^2,\quad x_{12}^2=x_{21}^2,\quad \chi _{21}^n=\lambda\chi _{12}^n,
\quad \chi _{11}^n=\chi _{22}^n,\\
&x_{11}^{2N}+\mu x_{12}^{2N}=1,\quad
x_{ij}x_{kl}=0\,\, \text{whenever $i+j+k+l$ is odd},
\end{align*}
where $\chi _{11}^m$, $\chi _{12}^m$, $\chi _{21}^m$ and $\chi _{22}^m$ are
defined as follows for  $m\in \Bbb Z^+$:
$$\chi _{11}^m:=\overbrace{x_{11}x_{22}x_{11}\ldots\ldots }^{\textrm{$m$ }},\quad \chi _{22}^m:=\overbrace{x_{22}x_{11}x_{22}\ldots\ldots }^{\textrm{$m$ }},$$
$$\chi _{12}^m:=\overbrace{x_{12}x_{21}x_{12}\ldots\ldots }^{\textrm{$m$ }},\quad \chi _{21}^m:=\overbrace{x_{21}x_{12}x_{21}\ldots\ldots }^{\textrm{$m$ }}.$$
The comultiplication, counit and antipode  of $A_{Nn}^{\mu\lambda}$ are given by
\begin{equation}\label{eq5.3}
\Delta (\chi_{ij}^k)=\chi_{i1}^k\otimes \chi_{1j}^k+\chi_{i2}^k\otimes \chi_{2j}^k,\quad
\varepsilon(x_{ij})=\delta_{ij}, \quad S(x_{ij})=x_{ji}^{4N-1},
\end{equation}
for $k\geq 1$, $i,j=1,2$. Let $\overline{i,i+j}=\{i,i+1,i+2,\cdots,i+j\}$ be  an index set.
Then the  basis of $A_{Nn}^{\mu\lambda}$ can be represented by
\begin{equation}\label{eq5.2}
\left\{x_{11}^s\chi _{22}^t,\ x_{12}^s\chi _{21}^t \mid %\ \vert \
s\in\overline{1,2N}, t\in\overline{0,n-1}
\right\}.
\end{equation}
Thus for $s,t\geq 0$ with $s+t\geq 1$,
\begin{align*}
\Delta (x_{11}^s\chi _{22}^t)
&=x_{11}^s\chi _{22}^t\otimes x_{11}^s\chi _{22}^t
+x_{12}^s\chi _{21}^t\otimes x_{21}^s\chi _{12}^t,\\
\Delta (x_{12}^s\chi _{21}^t)
&=x_{11}^s\chi _{22}^t\otimes x_{12}^s\chi _{21}^t
+x_{12}^s\chi _{21}^t\otimes x_{22}^s\chi _{11}^t.
\end{align*}

The cosemisimple Hopf algebra $A_{Nn}^{\mu\lambda}$ is decomposed to the direct sum of simple subcoalgebras such as
\[A_{Nn}^{\mu\lambda}
=\bigoplus_{g\in G}\k g\oplus\bigoplus_{
\substack{s\in\overline{1,N},\,\, t\in\overline{1,n-1}}}C_{st},
\]
see \cite[Theorem 3.1]{Suzuki1998} and \cite[Proposition 5.5]{Wakui2010a},
where
\begin{align*}
G&=\left\{x_{11}^{2s}\pm x_{12}^{2s}, x_{11}^{2s+1}\chi_{22}^{n-1}\pm
        \sqrt{\lambda}x_{12}^{2s+1}\chi_{21}^{n-1}\mid s\in\overline{1,N}\right\},\\
C_{st}&=\k x_{11}^{2s}\chi_{11}^t+\k x_{12}^{2s}\chi_{12}^t+
              \k x_{11}^{2s}\chi_{22}^t+\k x_{12}^{2s}\chi_{21}^t,\quad
              s\in\overline{1,N}, t\in\overline{1,n-1}. %1\leq t\leq n-1.
\end{align*}
The set $\left\{\k g\mid g\in G\right\}\cup \left\{\k x_{11}^{2s}\chi_{11}^t
+\k x_{12}^{2s}\chi_{21}^t \mid s\in\overline{1,N},
t\in\overline{1,n-1}
%1\leq t\leq n-1
\right\}$
is a full set of non-isomorphic simple left $A_{Nn}^{\mu\lambda}$-comodules, where
the coactions of the comodules listed above are given by the coproduct $\Delta$. Denote
the comodule $\k x_{11}^{2s}\chi_{11}^t
+\k x_{12}^{2s}\chi_{21}^t $ by $\Lambda_{st}$. That is to say the comodule
$\Lambda_{st}=\k w_1+\k w_2$
is defined as
\begin{align*}
\rho\left(w_1\right)
=  x_{11}^{2s}\chi_{11}^t\otimes w_1
     +x_{12}^{2s}\chi_{12}^t\otimes w_2,\quad
\rho\left(w_2\right)
=  x_{11}^{2s}\chi_{22}^t\otimes w_2
     +x_{12}^{2s}\chi_{21}^t\otimes w_1 .
\end{align*}
We define the support of $\Lambda_{st}$ as
\begin{align}
{\rm Supp}(\Lambda_{st})=\k x_{11}^{2s}\chi_{11}^t+\k x_{12}^{2s}\chi_{12}^t+
              \k x_{11}^{2s}\chi_{22}^t+\k x_{12}^{2s}\chi_{21}^t=C_{st}.
\end{align}

\section{Automorphism group of  Suzuki's  Hopf algebra}
\begin{lemma}\cite[Lemma 6.1]{MR1780094}\label{KeyLemma}
Let $H$ be a Hopf algebra, $\psi: H\rightarrow H$ an automorphism of Hopf algebras, $V$, $W$ Yetter-Drinfeld modules over $H$.
Let $V^\psi$ be the same space underlying $V$ but with action and coaction
$$h\cdot_\psi v=\psi(h)\cdot v,\quad \rho^\psi (v)
=\left(\psi^{-1}\otimes \mathrm{id}\right)\rho(v), \quad h\in H, v\in V.$$
Then $V^\psi$ is also a Yetter-Drinfeld module over $H$.
If $T: V\rightarrow W$ is a morphism in ${}_H^H\mathcal{YD}$, then
$T^\psi: V^\psi\rightarrow W^\psi$ also is. Moreover, the braiding
$c: V^\psi\otimes  W^\psi\rightarrow    W^\psi \otimes V^\psi$
coincides with the braiding $c: V\otimes W\rightarrow W\otimes V$.
\end{lemma}
\begin{remark}
We have ${\rm Supp} (V^\psi)=\psi^{-1}({\rm Supp}(V))$.
\end{remark}

\begin{lemma}
Let  $n=2m+1$ be odd and $\psi$ be any  Hopf algebra automorphism of  $A_{Nn}^{\mu\lambda}$, then
there exist some $s\in\overline{1,N}$ and  $t\in\overline{1,n-1}$ such that $\psi(C_{N1})=C_{st}$.
\end{lemma}
\begin{proof}
 According to \cite[Theorem 3.1 and Table 1]{Shi2020odd},
 there are exactly $8N^2m(m+1)$ pairwise  non-isomorphic Yetter-Drinfeld modules over
 $A_{N\, 2m+1}^{\mu\lambda}$ of dimension 2:
\begin{enumerate}
\item $\mathscr{C}_{jk,p}^{st}$, $s\in\overline{1,N}$, $t\in\overline{0,m-1}$,
    $\frac j2\in\overline{1,m}$,
    $k\in\overline{0,N-1}$, $p\in\Bbb{Z}_2$,
    $\mathscr{C}_{jk,p}^{st}\cong \Lambda_{s, 2t+2}$ as comdules;
\item $\mathscr{D}_{jk,p}^{st}$, $s\in\overline{1,N}$, $t\in\overline{0,m-1}$,
    $\frac j2\in\overline{1,m}$,
    $k\in\overline{0,N-1}$, $p\in\Bbb{Z}_2$,
    $\mathscr{D}_{jk,p}^{st}\cong \Lambda_{s, 2t+2}$ as comdules;
\item $\mathscr{E}_{jk,p}^{s}$, $s\in\overline{1,N}$,
    $\frac j2\in\overline{1,m}$,
    $k\in\overline{0,N-1}$, $p\in\Bbb{Z}_2$,
    $\mathscr{E}_{jk,p}^{s}\cong \k g_s^+\oplus \k g_s^-$ as comodules;
\item $\mathscr{F}_{k,p}^{st}$, $s\in\overline{1,N}$, $t\in\overline{0,m-1}$,
    $k\in\overline{0,N-1}$, $p\in\Bbb{Z}_2$,
    $\mathscr{F}_{k,p}^{st}\cong \Lambda_{s, 2t+2}$ as comodules;
\item $\mathscr{G}_{k,p}^{st}$, $s\in\overline{1,N}$, $t\in\overline{0,m-1}$,
    $k\in\overline{0,N-1}$, $p\in\Bbb{Z}_2$,
    $\mathscr{G}_{k,p}^{st}\cong \Lambda_{s, 2t+1}$ as comodules;
\item $\mathscr{H}_{jk,p}^{st}$, $s\in\overline{1,N}$, $t\in\overline{0,m-1}$,
    $\frac j2\in\overline{1,m}$,
    $k\in\overline{0,N-1}$, $p\in\Bbb{Z}_2$,
    $\mathscr{H}_{jk,p}^{st}\cong \Lambda_{s, 2t+1}$ as comodules;
\item $\mathscr{I}_{jk,p}^{st}$, $s\in\overline{1,N}$, $t\in\overline{0,m-1}$,
    $\frac j2\in\overline{1,m}$,
    $k\in\overline{0,N-1}$, $p\in\Bbb{Z}_2$,
    $\mathscr{I}_{jk,p}^{st}\cong \Lambda_{s, 2t+1} $ as comodules in this situation;
\item $\mathscr{I}_{jk,p}^{st}$, $s\in\overline{1,N}$, $t=m$,
    $\frac j2\in\overline{1,m}$,
    $k\in\overline{0,N-1}$, $p\in\Bbb{Z}_2$,
    $\mathscr{I}_{jk,p}^{st}\cong \k h_s^+\oplus \k h_s^-$ as comodules in this situation.
\end{enumerate}
Here $g_s^{\pm}=x_{11}^{2s}\pm x_{12}^{2s}$, $h_s^{\pm}=x_{11}^{2s+1}\chi_{22}^{2m}
\pm\sqrt{\lambda} x_{12}^{2s+1}\chi_{21}^{2m}$ are group-likes of $A_{N\, 2m+1}^{\mu\lambda}$.
Set $W=\mathscr{G}_{k,p}^{N0}$, then ${\rm Supp}(W)=C_{N1}=\k x_{11}+\k x_{12}+\k x_{21}+\k x_{22}$.
Since $C_{N1}$ does not contain any group-likes, so
$\psi(C_{N1})=\psi({\rm Supp}(W))={\rm Supp} (W^{\psi^{-1}})=C_{s\,\, 2t+1}$ or $C_{s\,\, 2t+2}$
for some $s\in\overline{1,N}$, $t\in\overline{0,m-1}$.
\end{proof}

\begin{lemma}
Let  $n=2m$ be even and $\psi$ be any  Hopf algebra automorphism of  $A_{Nn}^{\mu\lambda}$, then
there exist some $s\in\overline{1,N}$ and  $t\in\overline{1,n-1}$ such that $\psi(C_{N1})=C_{st}$.
\end{lemma}
\begin{proof}
According to \cite[Theorem 3.1 and Table 1]{Shi2020even},
there are exactly $2N^2(4m^2-1)$  non-isomorphic Yetter-Drinfeld modules
over $A_{N\, 2m}^{\mu\lambda}$ of dimension 2:
\begin{enumerate}
\item $\mathscr{B}_{01k}^s$, $s\in\overline{1,N}$, $k\in\overline{0,N-1}$,
           $\mathscr{B}_{01k}^s\cong \k g_s^+\oplus \k g_s^-$ as comodules;
\item $\mathscr{C}_{ijk,p}^{st}$, $ij=00$ or $01$, $k\in\overline{0,N-1}$, $p\in \Bbb{Z}_2$,
          $s\in\overline{1,N}$, $t\in\overline{0,m-2}$,
          $\mathscr{C}_{ijk,p}^{st}\cong \Lambda_{s\, 2t+2}$ as comodules in this situation;
\item $\mathscr{C}_{ijk,p}^{st}$, $i=0$,
          $j=\left\{\begin{array}{ll} i+1,&\text{if}\,\lambda=1,\\i,&\text{if}\,\lambda=-1,\end{array}\right.$
          $k\in\overline{0,N-1}$, $s\in\overline{1,N}$, $p=0$,
          $t=m-1$,
          $\mathscr{C}_{ijk,p}^{st}\cong \k h_s^+\oplus \k h_s^-$ as comodules in this situation;
\item $\mathscr{D}_{jk,p}^{st}$, $\frac j2\in\overline{1,m-1}$,
          $k\in\overline{0,N-1}$, $p\in \Bbb{Z}_2$, $s\in\overline{1, N}$, $t\in\overline{0,m-1}$,
          $\mathscr{D}_{jk,p}^{st}\cong \left\{\begin{array}{ll}
          \Lambda_{s\, 2t+2}, & t\neq m-1, \\
          \k h_s^+\oplus \k h_s^-, & t=m-1, \\
          \end{array}\right.$ as comodules;
\item $\mathscr{E}_{jk,p}^{st}$, $\frac j2\in\overline{1,m-1}$,
          $k\in\overline{0,N-1}$, $p\in \Bbb{Z}_2$, $s\in\overline{1,N}$, $t\in\overline{0,m-1}$,
          $\mathscr{E}_{jk,p}^{st}\cong \left\{\begin{array}{ll}
          \Lambda_{s\, 2t}, & t\neq 0, \\
          \k g_s^+\oplus \k g_s^-, & t=0, \\
          \end{array}\right.$ as comodules;
\item $\mathscr{G}_{jk,p}^{st}$,
          $\left\{\begin{array}{ll}\frac j2\in\overline{1,m-1}, &\text{if}\,\lambda=1,\vspace{1mm}\\
             \frac {j+1}2\in\overline{1,m}, &\text{if}\,\lambda=-1,
          \end{array}\right.$
          $k\in\overline{0,N-1}$, $p\in \Bbb{Z}_2$, $s\in\overline{1,N}$,
          $t\in\overline{0,m-1}$,
          $\mathscr{G}_{jk,p}^{st}\cong \Lambda_{s\, 2t+1}$ as comdules;
\item $\mathscr{H}_{jk,p}^{st}$,
          $\left\{\begin{array}{ll}\frac j2\in\overline{1,m-1}, &\text{if}\,\lambda=1,\vspace{1mm}\\
             \frac {j+1}2\in\overline{1,m}, &\text{if}\,\lambda=-1,
          \end{array}\right.$
          $k\in\overline{0,N-1}$, $p\in \Bbb{Z}_2$, $s\in\overline{1,N}$,
          $t\in\overline{0,m-1}$,
          $\mathscr{H}_{jk,p}^{st}\cong \Lambda_{s\, 2t+1}$ as comdules;
\item $\mathscr{P}_{ijk,p}^{st}$, $ij=00$ or $01$, $k\in\overline{0,N-1}$, $p\in \Bbb{Z}_2$,
          $s\in\overline{1,N}$, $t\in\overline{0,m-1}$,
          $\mathscr{P}_{ijk,p}^{st}\cong \Lambda_{s\, 2t+1}$ as comdules.
\end{enumerate}
Here $g^{\pm}_s=x_{11}^{2s}\pm x_{12}^{2s}$, $h^{\pm}_s=x_{11}^{2s}\chi_{11}^{2m}
\pm\sqrt{\lambda} x_{12}^{2s}\chi_{12}^{2m}$ are goup-likes of $A_{N\, 2m}^{\mu\lambda}$.
Set $W=\mathscr{G}_{k,p}^{N0}$, then
${\rm Supp}(W)=C_{N1}=\k x_{11}+\k x_{12}+\k x_{21}+\k x_{22}$.
Since $C_{N1}$ does not contain any group-likes, so
$\psi(C_{N1})=\psi({\rm Supp}(W))={\rm Supp} (W^{\psi^{-1}})=C_{st}$
for some $s\in\overline{1,N}$, $t\in\overline{1,n-1}$.
\end{proof}

Let $\phi$ be any Hopf algebra automorphism of  $A_{Nn}^{\mu\lambda}$, then $\phi(C_{N1})=C_{st}$
for some $s\in\overline{1,N}$, $t\in\overline{1,n-1}$. So we can suppose that
\begin{align}\label{DefinitionofPhi}
\left\{\begin{array}{l}
\phi(x_{11})=a_1x_{11}^{2s}\chi_{11}^t+(1-a_1)x_{11}^{2s}\chi_{22}^t
                       +a_2x_{12}^{2s}\chi_{12}^t+a_3x_{12}^{2s}\chi_{21}^t,\vspace{1mm}\\
\phi(x_{22})=b_1x_{11}^{2s}\chi_{11}^t+(1-b_1)x_{11}^{2s}\chi_{22}^t
                       +b_2x_{12}^{2s}\chi_{12}^t+b_3x_{12}^{2s}\chi_{21}^t,\vspace{1mm}\\
\phi(x_{12})=d_1x_{11}^{2s}\chi_{11}^t-d_1 x_{11}^{2s}\chi_{22}^t
                       +d_2x_{12}^{2s}\chi_{12}^t+d_3x_{12}^{2s}\chi_{21}^t,\vspace{1mm}\\
\phi(x_{21})=e_1x_{11}^{2s}\chi_{11}^t-e_1 x_{11}^{2s}\chi_{22}^t
                       +e_2x_{12}^{2s}\chi_{12}^t+e_3x_{12}^{2s}\chi_{21}^t,
\end{array}\right.
\end{align}
where $a_i$, $b_i$, $d_i$, $e_i\in\k$ for $i\in\overline{1,3}$.

%\newpage
\begin{lemma}
Let $\phi$ be an automorphism of $A_{Nn}^{\mu\lambda}$ as defined in \eqref{DefinitionofPhi}.
If  $t$ is odd, then
\begin{align}
\phi\left(x_{11}^2\right)=\phi\left(x_{22}^2\right)
&\Leftrightarrow \left\{\begin{array}{ll}
(a_1-b_1)(1-a_1-b_1)=0,&\\
%a_1^2+(1-a_1)^2=b_1^2+(1-b_1)^2, &\vspace{1mm}\\
a_2^2+a_3^2=b_2^2+b_3^2,&\\
%a_1(1-a_1)=b_1(1-b_1),&\\
a_2a_3=b_2b_3, &\text{if}\,\, t\neq \frac n2,\\
(1+\lambda)a_2a_3=(1+\lambda)b_2b_3, &\text{if}\,\, t= \frac n2.
\end{array}\right. \label{eq1}\\
%%%%%%%%%%%%%
\phi\left(x_{12}^2\right)=\phi\left(x_{21}^2\right)
&\Leftrightarrow \left\{\begin{array}{ll}
d_1^2=e_1^2, &\vspace{1mm}\\
d_2^2+d_3^2=e_2^2+e_3^2,&\\
d_2d_3=e_2e_3, &\text{if}\,\, t\neq \frac n2,\\
(1+\lambda)d_2d_3=(1+\lambda)e_2e_3, &\text{if}\,\, t= \frac n2.
\end{array}\right. \label{eq2}\\
%%%%%%%%%%%%%%%
\phi(x_{11}x_{12})=0
&\Leftrightarrow \left\{\begin{array}{ll}
(1-2a_1)d_1=0,\quad a_2d_2+a_3d_3=0,&\\
d_1=0,\quad a_2d_3=0,\quad a_3d_2=0,&\text{if}\,\, t\neq \frac n2,\vspace{1mm}\\
a_2d_3+\lambda a_3d_2=0, &\text{if}\,\, t=\frac n2.
\end{array}\right. \label{eq3}\\
%%%%%%%%%%%%%%%%
\phi(x_{22}x_{12})=0
&\Leftrightarrow \left\{\begin{array}{ll}
(1-2b_1)d_1=0,\quad b_2d_2+b_3d_3=0,&\\
d_1=0,\quad b_2d_3=0,\quad b_3d_2=0,&\text{if}\,\, t\neq \frac n2,\vspace{1mm}\\
b_2d_3+\lambda b_3d_2=0, &\text{if}\,\,t=\frac n2.
\end{array}\right. \label{eq4}\\
%%%%%%%%%%%%%%%
\phi(x_{11}x_{21})=0
&\Leftrightarrow \left\{\begin{array}{ll}
(1-2a_1)e_1=0,\quad a_2e_2+a_3e_3=0,&\\
e_1=0,\quad a_2e_3=0,\quad a_3e_2=0,&\text{if}\,\, t\neq \frac n2,\vspace{1mm}\\
a_2e_3+\lambda a_3e_2=0, &\text{if}\,\,t=\frac n2.
\end{array}\right. \label{eq5}\\
%%%%%%%%%%%%%%%%%
\phi(x_{22}x_{21})=0
&\Leftrightarrow \left\{\begin{array}{ll}
(1-2b_1)e_1=0,\quad b_2e_2+b_3e_3=0,&\\
e_1=0,\quad b_2e_3=0,\quad b_3e_2=0,&\text{if}\,\, t\neq \frac n2,\vspace{1mm}\\
b_2e_3+\lambda b_3e_2=0, &\text{if}\,\,t=\frac n2.
\end{array}\right. \label{eq6}
\end{align}
\end{lemma}
\begin{proof}
It is a direct verification.
\end{proof}

\begin{lemma}\label{tIsEven}
Let $\phi$ be an automorphism of $A_{Nn}^{\mu\lambda}$ as defined in \eqref{DefinitionofPhi}.
If  $t$ is even, then
\begin{align}
\phi\left(x_{11}^2\right)=\phi\left(x_{22}^2\right)
&\Leftrightarrow \left\{\begin{array}{ll}
(a_1-b_1)(1-a_1-b_1)=0, &\\
a_2a_3=b_2b_3,&\\
a_1=b_1,\quad a_2^2=b_2^2,\quad a_3^2=b_3^2,  &\text{if}\,\, t\neq \frac n2,\vspace{1mm}\\
a_2^2+\lambda a_3^2=b_2^2+\lambda b_3^2, &\text{if}\,\, t= \frac n2.
\end{array}\right. \label{eq7}\\
%%%%%%%%%%%%%
\phi\left(x_{12}^2\right)=\phi\left(x_{21}^2\right)
&\Leftrightarrow \left\{\begin{array}{ll}
d_1^2=e_1^2, &\\
d_2d_3=e_2e_3,&\\
d_2^2=e_2^2,\quad d_3^2=e_3^2,  &\text{if}\,\, t\neq \frac n2,\vspace{1mm}\\
d_2^2+\lambda d_3^2=e_2^2+\lambda e_3^2, &\text{if}\,\, t= \frac n2.
\end{array}\right. \label{eq8}\\
%%%%%%%%%%%%%%%
\phi(x_{11}x_{12})=0
&\Leftrightarrow \left\{\begin{array}{ll}
(1-2a_1)d_1=0,\quad a_2d_3+a_3d_2=0,&\\
d_1=0,\quad a_2d_2=0,\quad a_3d_3=0,&\text{if}\,\, t\neq \frac n2,\vspace{1mm}\\
a_2d_2+\lambda a_3d_3=0, &\text{if}\,\,t=\frac n2.
\end{array}\right.\label{eq9}\\
%%%%%%%%%%%%%%%%%
\phi(x_{22}x_{12})=0
&\Leftrightarrow \left\{\begin{array}{ll}
(1-2b_1)d_1=0,\quad b_2d_3+b_3d_2=0,&\\
d_1=0,\quad b_2d_2=0,\quad b_3d_3=0,&\text{if}\,\, t\neq \frac n2,\vspace{1mm}\\
b_2d_2+\lambda b_3d_3=0, &\text{if}\,\,t=\frac n2.
\end{array}\right.\label{eq10}\\
%%%%%%%%%%%%%%
\phi(x_{11}x_{21})=0
&\Leftrightarrow \left\{\begin{array}{ll}
(1-2a_1)e_1=0,\quad a_2e_3+a_3e_2=0,&\\
e_1=0,\quad a_2e_2=0,\quad a_3e_3=0,&\text{if}\,\, t\neq \frac n2,\vspace{1mm}\\
a_2e_2+\lambda a_3e_3=0, &\text{if}\,\,t=\frac n2.
\end{array}\right. \label{eq11}\\
%%%%%%%%%%%%%%%
\phi(x_{22}x_{21})=0
&\Leftrightarrow \left\{\begin{array}{ll}
(1-2b_1)e_1=0,\quad b_2e_3+b_3e_2=0,&\\
e_1=0,\quad b_2e_2=0,\quad b_3e_3=0,&\text{if}\,\, t\neq \frac n2,\vspace{1mm}\\
b_2e_2+\lambda b_3e_3=0, &\text{if}\,\,t=\frac n2.
\end{array}\right.\label{eq12}
\end{align}
\end{lemma}
\begin{proof}
It is a direct verification.
\end{proof}

\begin{lemma}
Let $\phi$ be an automorphism of $A_{Nn}^{\mu\lambda}$ as defined in \eqref{DefinitionofPhi}.
Then
\begin{align}
\Delta\phi\left(x_{11}\right)
=(\phi\otimes \phi)\Delta\left(x_{11}\right)
&\Leftrightarrow \left\{\begin{array}{ll}
a_1^2+d_1e_1=a_1,& a_2a_1+d_2e_1=0,\\
a_1a_2+d_1e_2=a_2,& a_2^2+d_2e_2=0,\\
a_1a_3+d_1e_3=0,& a_2a_3+d_2e_3=a_1,\\
a_3a_1+d_3e_1=a_3,& a_3a_2+d_3e_2=1-a_1,\\
a_3^2+d_3e_3=0.&
\end{array}\right. \label{eq13}\\
\Delta\phi\left(x_{12}\right)
=(\phi\otimes \phi)\Delta\left(x_{12}\right)
&\Leftrightarrow \left\{\begin{array}{ll}
a_1d_1+d_1b_1=d_1, & a_1d_2+d_1b_2=d_2,\\
a_1d_3+d_1b_3=0, & a_2d_1+d_2b_1=0,\\
a_2d_2+d_2b_2=0,&a_2d_3+d_2b_3=d_1,\\
a_3d_1+d_3b_1=d_3,&a_3d_2+d_3b_2=-d_1,\\
a_3d_3+d_3b_3=0.&
\end{array}\right. \label{eq14}\\
%\end{align}
%\end{lemma}
%\begin{lemma}
%\begin{align*}
\Delta\phi\left(x_{22}\right)
=(\phi\otimes \phi)\Delta\left(x_{22}\right)
&\Leftrightarrow \left\{\begin{array}{ll}
b_1^2+d_1e_1=b_1, & b_2b_3+d_3e_2=b_1,\\
b_2b_3+d_2e_3=1-b_1, & b_1b_2+d_1e_2=0,\\
b_1b_3+d_1e_3=b_3, &b_1b_2+d_2e_1=b_2,\\
b_1b_3+d_3e_1=0, & b_2^2+d_2e_2=0,\\
b_3^2+d_3e_3=0. &
\end{array}\right. %\label{eq15}
\\
\Delta\phi\left(x_{21}\right)
=(\phi\otimes \phi)\Delta\left(x_{21}\right)
&\Leftrightarrow \left\{\begin{array}{ll}
b_1e_1+e_1a_1=e_1, & b_2e_3+e_2a_3=e_1,\\
a_1e_2+b_2e_1=0, &b_1e_2+a_2e_1=e_2,\\
b_3e_1+a_1e_3=e_3, &b_3e_2+e_3a_2=-e_1,\\
b_1e_3+a_3e_1=0, & b_2e_2+e_2a_2=0,\\
b_3e_3+e_3a_3=0. &
\end{array}\right. \label{Delta_phi_x21}
\end{align}
\end{lemma}
\begin{proof}
It is a direct verification.
\end{proof}

%\newpage
\begin{theorem} \label{MainProof}
If $\phi$ is a  Hopf  algebra automorphism of $A_{Nn}^{\mu\lambda}$, then $\phi$
is one of  the  following automorphisms.
\begin{enumerate}
\item $\Psi_{d_2}^{s,t}$ is a Hopf algebra automorphism of $A_{Nn}^{\mu\lambda}$ such that
                   \begin{align*}
                   \Psi_{d_2}^{s,t}\left(x_{11}\right)&=x_{11}^{2s}\chi_{11}^t,\vspace{1mm} &
                   \Psi_{d_2}^{s,t}\left(x_{22}\right)&=x_{11}^{2s}\chi_{22}^t,\vspace{1mm}\\
                   \Psi_{d_2}^{s,t}\left(x_{12}\right)&=d_2x_{12}^{2s}\chi_{12}^t,\vspace{1mm}&
                   \Psi_{d_2}^{s,t}\left(x_{21}\right)&=d_2^{-1}x_{12}^{2s}\chi_{21}^t.
                   \end{align*}
        And  one of the following conditions is satisfied.
        \begin{enumerate}
        \item $t$ is odd, $t\neq \frac n2$,  $(N, 2s+t)=1=(t, n)$,
                  $\left\{\begin{array}{ll}
                  d_2^{2N}=1=d_2^4, &\text{$n$ is even},\vspace{1mm}\\
                  d_2^2=1, &  \text{$n$ is odd}.\\
                  \end{array}\right.$
        \item $t=\frac n2=1$, $d_2^{2N}=1=d_2^4$ and $(2s+1, N)=1$.
        \end{enumerate}
\item $\Phi_{d_3}^{s,t}$ is a Hopf algebra automorphism of $A_{Nn}^{\mu\lambda}$ such that
                   \begin{align*}
                   \Phi_{d_3}^{s,t}\left(x_{11}\right)&=x_{11}^{2s}\chi_{22}^t,\vspace{1mm}&
                   \Phi_{d_3}^{s,t}\left(x_{22}\right)&=x_{11}^{2s}\chi_{11}^t,\vspace{1mm}\\
                   \Phi_{d_3}^{s,t}\left(x_{12}\right)&=d_3x_{12}^{2s}\chi_{21}^t,\vspace{1mm}&
                   \Phi_{d_3}^{s,t}\left(x_{21}\right)&=d_3^{-1}x_{12}^{2s}\chi_{12}^t.
                   \end{align*}
        And  one of the following conditions is satisfied.
        \begin{enumerate}
        \item $t$ is odd, $t\neq \frac n2$,  $(N, 2s+t)=1=(t, n)$,
                  $\left\{\begin{array}{ll}
                  d_3^{2N}=1=d_3^4, &\text{$n$ is even},\vspace{1mm}\\
                  d_3^2=1, &  \text{$n$ is odd}.\\
                  \end{array}\right.$
        \item $t=\frac n2=1$, $d_3^{2N}=1=d_3^4$ and $(2s+1, N)=1$.
        \end{enumerate}
\item If n=2, $(2s+1, N)=1$
and $\left\{\begin{array}{ll}
\zeta_1^4=\zeta_2^4=\mu=1, &\text{$N$ even,}\vspace{1mm}\\
\zeta_1^2=\zeta_2^2=\mu, &\text{$N$  odd, }\\
\end{array}\right.$ then
$\Gamma^{\zeta_1,\zeta_2}_s$ is a Hopf algebra automorphism of $A_{Nn}^{\mu+}$ such that
             \begin{align}\label{DefAutoGamma}
             \left\{\begin{array}{l}
             \Gamma^{\zeta_1,\zeta_2}_s(x_{11})=\frac12\left[x_{11}^{2s}\left(x_{11}+x_{22}\right)
                             + x_{12}^{2s}\left(\zeta_2 x_{12}+\zeta_2^{-1}x_{21}\right)\right],\vspace{1mm}\\
             \Gamma^{\zeta_1,\zeta_2}_s(x_{22})=\frac12\left[x_{11}^{2s}\left(x_{11}+x_{22}\right)
                             -x_{12}^{2s} \left(\zeta_2x_{12}+\zeta_2^{-1}x_{21}\right)\right],\vspace{1mm}\\
             \Gamma^{\zeta_1,\zeta_2}_s(x_{12})=\frac{\zeta_1}2\left[x_{11}^{2s}\left(x_{11}-x_{22}\right)
                             - x_{12}^{2s} \left(\zeta_2x_{12}-\zeta_2^{-1}x_{21}\right)\right],\vspace{1mm}\\
            \Gamma^{\zeta_1,\zeta_2}_s(x_{21})=\frac{\zeta_1^{-1}}2\left[x_{11}^{2s}\left(x_{11}-x_{22}\right)
                             + x_{12}^{2s} \left(\zeta_2x_{12}-\zeta_2^{-1}x_{21}\right)\right].
            \end{array}\right.
            \end{align}
\end{enumerate}
\end{theorem}

%\newpage
\begin{proof}
%\begin{enumerate}
(1) Suppose $t$ is odd and $t\neq \frac n2$, then $d_1=e_1=0$. The first line of \eqref{eq13} implies that $a_1=0$ or $1$.
         \begin{enumerate}[(a)]
         \item Suppose $a_1=1$, then $a_2=a_3=0=d_2e_2$ from the formula \eqref{eq13}
                   and $d_3=0$ from the second line of  the formula \eqref{eq14}.
                   Since $\phi(x_{12})\neq 0$, $d_1=d_3=0$ implies that
                   $d_2\neq 0$, hence $e_2=0$.
                   Since $a_2=a_3=0$,
                   the second line and the third line of the formula \eqref{eq1}
                   implies that $b_2=b_3=0$.  Since $d_1=0\neq d_2$, the second line of the formula
                   \eqref{eq14} implies $b_1=0$.
                   The third line of the formula \eqref{eq13} implies $d_2e_3=1$.
                   Now we see that $\phi=\Psi_{d_2}^{s,t}$.
                   From the second line of \eqref{eq2}, we have $d_2^2=e_3^2$, so $d_2^4=1$.
                   It is easy to check that
                   \begin{align*}
                   \phi\left(x_{11}^{2k}+\mu x_{12}^{2k}\right)
                   =x_{11}^{2k(2s+t)}+\mu d_2^{2k}x_{12}^{2k(2s+t)}&\Rightarrow
                   (N, 2s+t)=1,  d_2^{2N}=1,\\
                   S\phi=\phi S&\Rightarrow d_2^{4N}=1.
                   \end{align*}
                   If $k$ is even, then
                   \[
                   \phi(\chi_{12}^k)=x_{12}^{2ks}\chi_{12}^{kt},\quad
                   \lambda \phi(\chi_{21}^k)=\lambda x_{12}^{2ks}\chi_{21}^{kt},
                   \]
                   which implies that $(t, n)=1$ in case that $n$ is even. \\
                   If $k$ is odd, then
                   \[
                   \phi(\chi_{12}^k)=d_2 x_{12}^{2ks}\chi_{12}^{kt},\quad
                   \lambda \phi(\chi_{21}^k)=\lambda d_2^{-1}x_{12}^{2ks}\chi_{21}^{kt},
                   \]
                   which implies that $(t, n)=1$ and $d_2^2=1$ in case that $n$ is odd.

         \item  Suppose $a_1=0$,  then $a_2=a_3=0$, $d_3e_2=1$, $d_2=0=e_3$
                    from the formula  \eqref{eq13}.
                    It implies  $b_2=b_3=0$ from the second line and third line of \eqref{eq1}.
                    According to the fourth line of formula \eqref{eq14}, $b_1=1$. Hence
                    $\phi=\Phi_{d_3}^{s,t}$.
                   From the second line of \eqref{eq2}, we have $d_3^2=e_2^2$, so $d_3^4=1$.
                   It is easy to check that
                   \begin{align*}
                   \phi\left(x_{11}^{2k}+\mu x_{12}^{2k}\right)
                   =x_{11}^{2k(2s+t)}+\mu d_3^{2k}x_{12}^{2k(2s+t)}
                   &\Rightarrow
                   (N, 2s+t)=1,  d_3^{2N}=1,\\
                   S\phi=\phi S &\Rightarrow d_3^{4N}=1.
                   \end{align*}
                   If $k$ is even, then
                   \[
                   \phi(\chi_{12}^k)=x_{12}^{2ks}\chi_{21}^{kt},\quad
                   \lambda \phi(\chi_{21}^k)=\lambda x_{12}^{2ks}\chi_{12}^{kt},
                   \]
                   which implies that $(t, n)=1$ in case that $n$ is even. \\
                   If $k$ is odd, then
                   \[
                   \phi(\chi_{12}^k)=d_3 x_{12}^{2ks}\chi_{21}^{kt},\quad
                   \lambda \phi(\chi_{21}^k)=\lambda d_3^{-1}x_{12}^{2ks}\chi_{12}^{kt},
                   \]
                   which implies that $(t, n)=1$ and $d_3^2=1$ in case that $n$ is odd.
         \end{enumerate}
(2)  Suppose $t$ is odd and $t= \frac n2$. If $a_1\neq \frac12$,
then $d_1=e_1=0$ according to \eqref{eq9} and \eqref{eq11}.
          The first line of the formula
         \eqref{eq13} implies that $a_1=0$ or $1$.
         \begin{enumerate}[(a)]
         \item Suppose that $a_1=1$, then it implies that $a_2=a_3=0=d_2e_2$,
         $d_2e_3=1$, $d_3e_2=0$, $d_3e_3=0$ from \eqref{eq13}.
         Since $d_2e_3=1\neq 0$, we have $e_2=d_3=0$.
         The formula \eqref{eq14} implies that $b_1=b_2=b_3=0$.
         Now we have  $\phi=\Psi_{d_2}^{s,t}$.
         From \eqref{eq2}, we obtain $d_2^2=e_3^2$ which implies $d_2^4=1$.
                 Since $\phi(x_{11}x_{22})=\phi(x_{22}x_{11})$ and
                 $\phi(x_{12}x_{21})=\phi(\lambda x_{21}x_{12})$,
                 we have $n=2$.  Let $\theta=\pm 1$, and $k\in\overline{1,N}$, then
                 \begin{align*}
                 \phi\left(x_{11}^2+\theta x_{12}^2\right)
                 &=x_{11}^{2(2s+1)}+\theta d_2^2 x_{12}^{2(2s+1)},\\
                 \phi\left(x_{11}^{2k}+\mu x_{12}^{2k}\right)
                 &= \phi\left[\left(x_{11}^{2}+x_{12}^{2}\right)^{k-1} \left(x_{11}^2+\mu x_{12}^2\right)\right]
                 =x_{11}^{2(2s+1)k}+\mu d_2^{2k}x_{12}^{2(2s+1)k}.
                 \end{align*}
                 So $d_2^{2N}=1$ and $(2s+1, N)=1$.
                 $\phi S=S\phi\Rightarrow d_2^{4N}=1$.
         \item Suppose that $a_1=0$, it is easy to see that   $\phi=\Phi_{d_3}^{s,t}$ and $d_3^4=1$.
                 Since $\phi(x_{11}x_{22})=\phi(x_{22}x_{11})$ and
                 $\phi(x_{12}x_{21})=\phi(\lambda x_{21}x_{12})$,
                 we have $n=2$.  Let $\theta=\pm 1$, and $k\in\overline{1,N}$, then
                 \begin{align*}
                 \phi\left(x_{11}^2+\theta x_{12}^2\right)&=x_{11}^{4s+2}+\theta d_3^2 x_{12}^{4s+2},\\
                 \phi\left(x_{11}^{2k}+\mu x_{12}^{2k}\right)
                 &= \phi\left[\left(x_{11}^{2}+x_{12}^{2}\right)^{k-1} \left(x_{11}^2+\mu x_{12}^2\right)\right]
                 =x_{11}^{2(2s+1)k}+\mu d_3^{2k}x_{12}^{2(2s+1)k}.
                 \end{align*}
                 So $d_3^{2N}=1$ and $(2s+1, N)=1$.
                 $\phi S=S\phi\Rightarrow d_3^{4N}=1$.
         \item Suppose that $a_1=\frac12$, then $b_1=\frac12$ from the formula \eqref{eq1}.
                  The first line of the formula \eqref{eq13} implies $d_1e_1=a_1-a_1^2=\frac 14$.
                  Since $d_1^2=e_1^2$,
%                  $d_1=e_1=\pm \frac12$  or $d_1=-e_1=\pm \frac{\sqrt{-1}}2$.
%                  \par\noindent
                  we can set $d_1=\frac{\zeta_1}2, e_1=\frac{\zeta_1^{-1}}2$, $\zeta_1^4=1$,
                  Then the formulas \eqref{eq13} implies that
                  \[
                  a_2=-d_2\zeta_1^{-1}=e_2\zeta_1,\quad
                  a_3=-e_3\zeta_1=d_3\zeta_1^{-1},\quad a_2a_3=\frac14.
                  \]
                  According to  the formula \eqref{eq14}, we have $d_2=b_2\zeta_1$
                  and $d_3=-b_3\zeta_1$.
                  From the first line of the formula  \eqref{eq3}, we can deduce $a_2^2=a_3^2$,
                  which implies that $a_2=\frac{\zeta_2}2, a_3=\frac{\zeta_2^{-1}}{2}$,
                  $\zeta_2^4=1$.   Now we have
                  \begin{align*}
                  a_1&=\frac12, &a_2&=\frac{\zeta_2}2, &a_3&=\frac{\zeta_2^{-1}}2,\\
                  b_1&=\frac12, &b_2&=-\frac{\zeta_2}2, &b_3&=-\frac{\zeta_2^{-1}}2,\\
                  d_1&=\frac{\zeta_1}2,&
                  d_2&=-\frac{\zeta_1\zeta_2}2, &
                  d_3&=\frac{\zeta_1\zeta_2^{-1}}2,\\
                  e_1&=\frac{\zeta_1^{-1}}2,&
                  e_2&=\frac{\zeta_1^{-1}\zeta_2}2, &
                  e_3&=-\frac{\zeta_1^{-1}\zeta_2^{-1}}2.
                  \end{align*}
                  The identity $a_2d_3+\lambda a_3d_2=0$ in \eqref{eq3} implies $\lambda=1$.
                  By direct computation, we have
                  $\phi(x_{11}x_{22}-x_{22}x_{11})=0$ and $\phi(x_{12}x_{21}-x_{21}x_{12})=0$
                  which implies $n=2$ and $t=\frac{n}2=1$.
                  So $\phi=\Gamma^{\zeta_1,\zeta_2}_s$.
                  According to Lemma \ref{N2},  we have   $(2s+1, N)=1$
                  and $\left\{\begin{array}{ll}
                  \mu=1, &\text{if $N$ is even,}\vspace{1mm}\\
                  \zeta_1^2=\zeta_2^2=\mu, &\text{if $N$ is odd.}\\
                  \end{array}\right.$
         \end{enumerate}
 \par \noindent
(3) Suppose $t$ is even and $t\neq \frac n2$, then $a_1=b_1$, $d_1=e_1=0$ from Lemma \ref{tIsEven}.
         The first line of the formulas
         \eqref{eq13} implies that $a_1=0$ or $1$.  If $a_1=0$, then $a_2=a_3=0$ from
         the formulas \eqref{eq13}. So $e_1=e_2=e_3=0$ from \eqref{Delta_phi_x21}. It is a contradiction since
         $\phi(x_{21})=0$.  If $a_1=1$, then $a_2=a_3=0$ from  the formulas \eqref{eq13}.
         So $b_1=1$, $b_2=b_3=0$ from \eqref{eq7}. It is  a contradiction since $\phi(x_{11})=\phi(x_{22})$.
 \par\noindent
(4) Suppose $t$ is even and $t= \frac n2$. If $a_1\neq \frac12$, then $d_1=e_1=0$ from
        the formulas \eqref{eq9} and \eqref{eq11}.
         The first line of the formulas
         \eqref{eq13} implies that $a_1=0$ or $1$.
         \begin{enumerate}[(a)]
         \item Suppose $a_1=1$. From formulas \eqref{eq13} and \eqref{eq14}, it is easy to see that
                  $\phi=\Psi_{d_2}^{s,t}$.
                 Since $\phi(x_{11}x_{22})=\phi(x_{22}x_{11})$ and $\phi(x_{12}x_{21})=\phi(x_{21}x_{12})$,
                 we have $n=2$ and $\lambda=1$. It is a contradiction with that $t=\frac{n}2$ is even.
         \item Suppose $a_1=0$. From formulas \eqref{eq13} and \eqref{eq14}, it is easy to see that
                   $\phi=\Phi_{d_3}^{s,t}$.
                 Since $\phi(x_{11}x_{22})=\phi(x_{22}x_{11})$ and $\phi(x_{12}x_{21})=\phi(x_{21}x_{12})$,
                 we have $n=2$ and $\lambda=1$. It is a contradiction with that $t=\frac{n}2$ is even.
         \item Suppose $a_1=\frac12$, then $b_1=\frac{1}2$ from the first line of \eqref{eq7}.
                  The first line of the formulas \eqref{eq13} implies $d_1e_1=a_1-a_1^2=\frac 14$.
                  Since $d_1^2=e_1^2$, we have $d_1=\frac{\zeta_1}2, e_1=\frac{\zeta_1^{-1}}{2}$, $\zeta_1^4=1$.
                  Then the formulas \eqref{eq13} implies that
                  \[
                  a_2=-d_2\zeta_1^{-1}=e_2\zeta_1,\quad
                  a_3=-e_3\zeta_1=d_3\zeta_1^{-1},\quad a_2a_3=\frac14.
                  \]
                  According to  the formula \eqref{eq14}, we have $d_2=b_2\zeta_1$
                  and $d_3=-b_3\zeta_1$.
                  From the third line of the formula  \eqref{eq9}, we can deduce $a_2^2=\lambda a_3^2$,
                  which implies that $a_2=\frac{\zeta_2}2, a_3=\frac{\zeta_2^{-1}}{2}$, $\zeta_2^4=\lambda$.
                  Now we see
                  $\phi=\Gamma^{\zeta_1,\zeta_2}_s$ as defined in  \eqref{DefAutoGamma}
                  with $\zeta_1^4=1$, $\zeta_2^4=\lambda$.
                 Since $\phi(x_{11}x_{22})=\phi(x_{22}x_{11})$ and $\phi(x_{12}x_{21})=\phi(x_{21}x_{12})$,
                 we have $n=2$. It is a contradiction with that $t=\frac{n}2$ is even.
         \end{enumerate}
%\end{enumerate}
\end{proof}

\begin{lemma}\label{N2}
Set $\zeta_1^4=\zeta_2^4=1$ and suppose
$\phi=\Gamma^{\zeta_1,\zeta_2}_s$ as defined in \eqref{DefAutoGamma} is a Hopf  algebra automorphism of $A_{N2}^{\mu+}$,
then $(2s+1, N)=1$
and $\left\{\begin{array}{ll}
\mu=1, &\text{if $N$ is even,}\vspace{1mm}\\
\zeta_1^2=\zeta_2^2=\mu, &\text{if $N$ is odd.}\\
\end{array}\right.$
\end{lemma}
\begin{proof}
Set $\theta=\pm 1$, $k\in\overline{1,N}$, then
\begin{align*}
\phi\left(x_{11}^2+\theta x_{12}^2\right)
&=\frac14x_{11}^{4s}
     \left[\left(x_{11}+x_{22}\right)^2+\theta \zeta_1^2(x_{11}-x_{22})^2\right]\\
&\quad+\frac14x_{12}^{4s}
     \left[\left(\zeta_2x_{12}+ \zeta_2^{-1}x_{21}\right)^2
     +\theta \zeta_1^2\left(\zeta_2x_{12}-\zeta_2^{-1}x_{21}\right)^2\right]\\
&=\frac{x_{11}^{4s}}2\left[(1+\theta\zeta_1^2)x_{11}^2+(1-\theta\zeta_1^2)x_{11}x_{22}\right]\\
&\quad+\frac{x_{12}^{4s}}2\left[(1+\theta\zeta_1^2)\zeta_2^2x_{12}^2+(1-\theta\zeta_1^2)x_{12}x_{21}\right]\\
&=x_{11}^{4s}\left[\frac{1+\theta\zeta_1^2}2x_{11}^2+\frac{1-\theta\zeta_1^2}2x_{11}x_{22}\right]
\\
&\quad
    +x_{12}^{4s}
     \left[\frac{1+\theta\zeta_1^2}2\zeta_2^2x_{12}^2+\frac{1-\theta\zeta_1^2}2x_{12}x_{21}\right]\\
&=\left\{\begin{array}{ll}
     x_{11}^{4s+2}+\zeta_2^2 x_{12}^{4s+2}, &\zeta_1^2=\theta,\vspace{1mm}\\
     x_{11}^{4s+1}x_{22}+x_{12}^{4s+1}x_{21}, &\zeta_1^2=-\theta.\\
     \end{array}\right.\\
%%%%%%%%%%%%%%%%%%
\phi\left(x_{11}^4+\theta x_{12}^4\right)
&=x_{11}^{8s}\left[\frac{1+\theta\zeta_1^2}2x_{11}^2+\frac{1-\theta\zeta_1^2}2x_{11}x_{22}\right]
     \left[\frac{1+\zeta_1^2}2x_{11}^2+\frac{1-\zeta_1^2}2x_{11}x_{22}\right]\\
&\quad+x_{12}^{8s}\left[\frac{1+\theta\zeta_1^2}2\zeta_2^2x_{12}^2
      +\frac{1-\theta\zeta_1^2}2x_{12}x_{21}\right]
 \\
 &\quad\cdot
     \left[\frac{1+\zeta_1^2}2\zeta_2^2x_{12}^2+\frac{1-\zeta_1^2}2x_{12}x_{21}\right]\\
&=\left\{\begin{array}{ll}
     x_{11}^{4(2s+1)}+x_{12}^{4(2s+1)}, &\theta=1,\vspace{1mm}\\
     x_{11}^{8s+3}x_{22}+\zeta_2^2 x_{12}^{8s+3}x_{21}, &\theta=-1.\\
     \end{array}\right.  \\
%%%%%%%%%%%%%%%%%%
\phi\left(x_{11}^{2k}+\mu x_{12}^{2k}\right)
&=\left\{\begin{array}{ll}
     x_{11}^{(4m+2)(2s+1)}+\zeta_2^2 x_{12}^{(4m+2)(2s+1)}, &\zeta_1^2=\mu, k=2m+1,\vspace{1mm}\\
     x_{11}^{4m(2s+1)+4s+1}x_{22}+x_{12}^{4m(2s+1)+4s+1}x_{21}, &\zeta_1^2=-\mu, k=2m+1,\vspace{1mm}\\
     x_{11}^{4m(2s+1)}+x_{12}^{4m(2s+1)}, &\mu=1, k=2m, \vspace{1mm}\\
     x_{11}^{4m(2s+1)-1}x_{22}+\zeta_2^2 x_{12}^{4m(2s+1)-1}x_{21}, &\mu=-1, k=2m.\\
     \end{array}\right.
\end{align*}
So $1=x_{11}^{2N}+\mu x_{12}^{2N}$ implies that  $(2s+1, N)=1$
and $\left\{\begin{array}{ll}
\mu=1, &\text{if $N$ is even,}\\
\zeta_1^2=\zeta_2^2=\mu, &\text{if $N$ is odd.}\\
\end{array}\right.$

As for any $l\in\Bbb{Z}^+$, by induction, we have
\[
\left(x_{11}+\theta x_{22}\right)^{l+1}
=(2x_{11})^l(x_{11}+\theta x_{22}),\quad
\left(\zeta_2 x_{12} +\zeta_2^{-1}x_{21} \right)^{l+1}
=(2\zeta_2)^l (\zeta_2 x_{12} +\zeta_2^{-1}x_{21}).
\]
%\begin{align*}
%\left(x_{11}+\theta x_{22}\right)^{2l+1}
%&=2^{2l}x_{11}^{2l}\left(x_{11}+\theta x_{22}\right),\\
%\left(\zeta_2 x_{12} +\theta\zeta_2^{-1}x_{21} \right)^2
%&=2x_{12}(\zeta_2^2x_{12}+\theta x_{21})
%=2\zeta_2x_{12}\left(\zeta_2 x_{12} +\theta \zeta_2^{-1}x_{21} \right),\\
%\left(\zeta_2 x_{12} +\theta\zeta_2^{-1}x_{21} \right)^{2l}
%&=(2\zeta_2x_{12})^{2l-1}\left(\zeta_2 x_{12} +\theta \zeta_2^{-1}x_{21} \right),\\
%%%%%%%%%%%%%%%
%\left(\zeta_2 x_{12} +\zeta_2^{-1}x_{21} \right)^{4N-1}
%&=(2\zeta_2 x_{12})^{4N-2}(\zeta_2 x_{12}+\zeta_2^{-1} x_{21}), \\
%\end{align*}
\begin{align*}
\phi S(x_{11})&=\phi\left(x_{11}^{4N-1}\right)\\
&=\frac1{2^{4N-1}}\left[x_{11}^{2s(4N-1)}\left(x_{11} +x_{22} \right)^{4N-1}
       +x_{12}^{2s(4N-1)}\left(\zeta_2 x_{12} +\zeta_2^{-1}x_{21} \right)^{4N-1}\right]\\
&=\frac1{2}\left[x_{11}^{2s(4N-1)+4N-2}\left(x_{11} +x_{22} \right)
       +\zeta_2^2 x_{12}^{2s(4N-1)+4N-2}\left(\zeta_2 x_{12} +\zeta_2^{-1}x_{21} \right)\right]\\
&=S\phi(x_{11}),\\
%%%%%%%%%%%%%%%
%\phi S(x_{22})&=\phi\left(x_{22}^{4N-1}\right)\\
%&=\frac1{2^{4N-1}}\left[x_{11}^{2s(4N-1)}\left(x_{11} +x_{22} \right)^{4N-1}
%       -x_{12}^{2s(4N-1)}\left(\zeta_2 x_{12} +\zeta_2^{-1}x_{21} \right)^{4N-1}\right]\\
%&=\frac1{2}\left[x_{11}^{2s(4N-1)+4N-2}\left(x_{11} +x_{22} \right)
%       -\zeta_2^2 x_{12}^{2s(4N-1)+4N-2}\left(\zeta_2 x_{12} +\zeta_2^{-1}x_{21} \right)\right]\\
%&=S\phi(x_{22}),\\
%%%%%%%%%%%%%%
\phi S(x_{12})&=\phi\left(x_{21}^{4N-1}\right)\\
&=\frac{\zeta_1^{-(4N-1)}}{2^{4N-1}}\left[x_{11}^{2s(4N-1)}\left(x_{11}-x_{22} \right)^{4N-1}
       +x_{12}^{2s(4N-1)}\left(\zeta_2 x_{12} -\zeta_2^{-1}x_{21} \right)^{4N-1}\right]\\
&=\frac{\zeta_1}2\left[x_{11}^{2s(4N-1)+4N-2}\left(x_{11}-x_{22} \right)
       +\zeta_2^{2}x_{12}^{2s(4N-1)+4N-2}\left(\zeta_2 x_{12} -\zeta_2^{-1}x_{21} \right)\right]\\
&=S\phi(x_{12}).
\end{align*}
Similarly, $\phi S(x_{22})=S\phi(x_{22})$,  $\phi S(x_{21})=S\phi(x_{21})$.
\end{proof}

\begin{lemma}
Let $\Psi_{\xi_1}^{s_1, t_1}, \Psi_{\xi_2}^{s_2, t_2},
\Phi_{\xi_1}^{s_1, t_1}, \Phi_{\xi_2}^{s_2, t_2}$ be  automorphisms of $A_{Nn}^{\mu\lambda}$
as defined in Theorem \ref{MainProof}, then
\begin{align}\label{AutRelationsMain}
\left\{\begin{array}{l}
\Psi_{\xi_1}^{s_1, t_1}\Psi_{\xi_2}^{s_2, t_2}
=\Psi_{\xi_1^{2s_2+1}\xi_2}^{s_1(2s_2+t_2)+s_2t_1,\, t_1t_2},\quad
\Phi_{\xi_1}^{s_1, t_1}\Phi_{\xi_2}^{s_2, t_2}
=\Psi_{\xi_1^{2s_2-1}\xi_2}^{s_1(2s_2+t_2)+s_2t_1,\, t_1t_2},\vspace{1mm}\\
\Psi_{\xi_1}^{s_1, t_1}\Phi_{\xi_2}^{s_2, t_2}
=\Phi_{\xi_1^{2s_2-1}\xi_2}^{s_1(2s_2+t_2)+s_2t_1,\, t_1t_2},\quad
\Phi_{\xi_1}^{s_1, t_1}\Psi_{\xi_2}^{s_2, t_2}
=\Phi_{\xi_1^{2s_2+1}\xi_2}^{s_1(2s_2+t_2)+s_2t_1,\, t_1t_2}.
\end{array}\right.
\end{align}
\end{lemma}
\begin{proof}
It is a direct computation. Here we only calculate
$\Psi_{\xi_1}^{s_1, t_1}\Psi_{\xi_2}^{s_2, t_2}(x_{12})$ and
$\Phi_{\xi_1}^{s_1, t_1}\Phi_{\xi_2}^{s_2, t_2}(x_{12})$.
\begin{align*}
\Psi_{\xi_1}^{s_1, t_1}\Psi_{\xi_2}^{s_2, t_2}(x_{12})
&=\Psi_{\xi_1}^{s_1, t_1}\left(\xi_2 x_{12}^{2s_2}\chi_{12}^{t_2}\right)
=\xi_2\left(\xi_1 x_{12}^{2s_1}\chi_{12}^{t_1}\right)^{2s_2+1}
   \left(x_{12}^{2s_1}\chi_{21}^{t_1}x_{12}^{2s_1}\chi_{12}^{t_1}\right)^{\frac{t_2-1}2}\\
&=\xi_2\xi_1^{2s_2+1}x_{12}^{2s_1(2s_2+t_2)+2s_2t_1}\chi_{12}^{t_1t_2},\\
&=\Psi_{\xi_1^{2s_2+1}\xi_2}^{s_1(2s_2+t_2)+s_2t_1,\, t_1t_2}(x_{12}),\\
%%%%%%%%%%%%%%%%%
\Phi_{\xi_1}^{s_1, t_1}\Phi_{\xi_2}^{s_2, t_2}(x_{12})
&=\Phi_{\xi_1}^{s_1, t_1}\left(\xi_2 x_{12}^{2s_2}\chi_{21}^{t_2}\right)
=\xi_2 \left(\xi_1 x_{12}^{2s_1}\chi_{21}^{t_1}\right)^{2s_2-1}
   \left(x_{12}^{2s_1}\chi_{21}^{t_1}x_{12}^{2s_1}\chi_{12}^{t_1}\right)^{\frac{t_2+1}2}\\
&=\xi_2\xi_1^{2s_2-1}x_{12}^{2s_1(2s_2+t_2)+2s_2t_1}\chi_{12}^{t_1t_2}\\
&=\Psi_{\xi_1^{2s_2-1}\xi_2}^{s_1(2s_2+t_2)+s_2t_1,\, t_1t_2}(x_{12}).
\end{align*}
\end{proof}
\begin{lemma}
Let $\Gamma^{a_1, a_2}_{s_1}, \Gamma^{b_1, b_2}_{s_2}$ be two automorphisms of $A_{N2}^{\mu+}$
as defined in Theorem \ref{MainProof}, then
\begin{align}\label{AutRelationsN2_1}
\Gamma^{a_1, a_2}_{s_1}\Gamma^{b_1, b_2}_{s_2}
&=\left\{\begin{array}{ll}
     \Psi_{a_2^{2s_2+1}b_1}^{s_1(2s_2+1)+s_2,\, 1},
     &a_1^{2s_2}a_1b_2=1, (a_1b_2)^2=1,\vspace{1mm}\\
     \Phi_{a_2^{2s_2-1}b_1}^{s_1(2s_2+1)+s_2,\, 1},
     &a_1^{2s_2}a_1b_2=-1,  (a_1b_2)^2=1,\vspace{1mm}\\
     \Gamma_{s_1(2s_2+1)+s_2}^{-b_1a_1^{2s_2}a_1b_2,\, -(a_1a_2)^{2s_2+1}b_2}, &(a_1b_2)^2=-1.\\
     \end{array}\right.
\end{align}
\end{lemma}
\begin{proof}
It is a direct computation. Here we only calculate $\Gamma^{a_1, a_2}_{s_1}\Gamma^{b_1, b_2}_{s_2}(x_{12})$.
\begin{align*}
&\quad\Gamma^{a_1, a_2}_{s_1}\Gamma^{b_1, b_2}_{s_2}(x_{12})\\
&=\frac{b_1}2\Gamma^{a_1, a_2}_{s_1}
     \left(x_{11}^{2s_2}(x_{11}-x_{22})-x_{12}^{2s_2}(b_2x_{12}-b_2^{-1}x_{21})\right)\\
&=\frac{b_1}2\left\{\frac1{2^{2s_2}}
      \left(x_{11}^{2s_1}(x_{11}+x_{22})+x_{12}^{2s_1}(a_2x_{12}+a_2^{-1}x_{21})\right)^{2s_2}
      x_{12}^{2s_1}(a_2x_{12}+a_1^{-1}x_{21})\right.\\
&\quad-\frac{a_1^{2s_2}}{2^{2s_2}}
     \left(x_{11}^{2s_1}(x_{11}-x_{22})-x_{12}^{2s_1}(a_2x_{12}-a_2^{-1}x_{21})\right)^{2s_2}\\
&\quad \left.\frac{b_2a_1}2
    \left[x_{11}^{2s_1}(x_{11}-x_{22})(1-(b_2a_1)^{-2})
    +x_{12}^{2s_1}(a_2x_{12}-a_2^{-1}x_{21})(-1-(b_2a_1)^{-2})
    \right]\right\}\\
%&=\left\{\begin{array}{ll}
%     \frac{a_2^{2s_2}b_1}2 x_{12}^{2s_1(2s_2+1)+2s_2}\left[(a_2x_{12}+a_2^{-1}x_{21})
%     +a_1^{2s_2}a_1b_2 (a_2x_{12}-a_2^{-1}x_{21})\right],
%     &(a_1b_2)^2=1,\vspace{1mm}\\
%     \frac{a_2^{2s_2}b_1}2x_{12}^{2s_1(2s_2+1)+2s_2}(a_2x_{12}+a_2^{-1}x_{21})
%     -\frac{b_1a_1^{2s_2}a_1b_2}2x_{11}^{2s_1(2s_2+1)+2s_2}(x_{11}-x_{22}),
%     &(a_1b_2)^2=-1,\\
%     \end{array}\right. \\
&=\left\{\begin{array}{ll}
     a_2^{2s_2}b_1a_2 x_{12}^{2s_1(2s_2+1)+2s_2}x_{12},
     &a_1^{2s_2}a_1b_2=1, (a_1b_2)^2=1,\vspace{1mm}\\
     a_2^{2s_2}b_1a_2^{-1} x_{12}^{2s_1(2s_2+1)+2s_2}x_{21},
     &a_1^{2s_2}a_1b_2=-1, (a_1b_2)^2=1,\vspace{1mm}\\
     \Gamma_{s_1(2s_2+1)+s_2}^{-b_1a_1^{2s_2}a_1b_2,\, -(a_1a_2)^{2s_2+1}b_2}(x_{12}), &(a_1b_2)^2=-1,\\
     \end{array}\right. \\
&=\left\{\begin{array}{ll}
     \Psi_{a_2^{2s_2+1}b_1}^{s_1(2s_2+1)+s_2,\, 1}(x_{12}),
     &a_1^{2s_2}a_1b_2=1, (a_1b_2)^2=1,\vspace{1mm}\\
     \Phi_{a_2^{2s_2-1}b_1}^{s_1(2s_2+1)+s_2,\, 1}(x_{12}),
     &a_1^{2s_2}a_1b_2=-1, (a_1b_2)^2=1,\vspace{1mm}\\
     \Gamma_{s_1(2s_2+1)+s_2}^{-b_1a_1^{2s_2}a_1b_2,\, -(a_1a_2)^{2s_2+1}b_2}(x_{12}), &(a_1b_2)^2=-1.\\
     \end{array}\right.
\end{align*}
\end{proof}

\begin{lemma}
Let $\Psi_{\xi}^{s_2, 1}, \Gamma^{\zeta_1, \zeta_2}_{s_1}$ be two automorphisms of $A_{N2}^{\mu+}$
as defined in Theorem \ref{MainProof}, then
\begin{align}\label{AutRelationsN2_2}
\Gamma^{\zeta_1, \zeta_2}_{s_1}\Psi_{\xi}^{s_2, 1}
=\Gamma^{\zeta_1^{2s_2+1}\xi,\, \zeta_2^{2s_2+1}}_{s_1(2s_2+1)+s_2},\quad
\Psi_{\xi}^{s_2, 1}\Gamma^{\zeta_1, \zeta_2}_{s_1}
=\Gamma^{\zeta_1,\, \xi^{2s_1+1} \zeta_2}_{s_1(2s_2+1)+s_2}.
\end{align}
\end{lemma}
\begin{proof}
It is a direct computation. Here we only calculate
$\Gamma^{\zeta_1, \zeta_2}_{s_1}\Psi_{\xi}^{s_2, 1}(x_{12})$ and
$\Psi_{\xi}^{s_2, 1}\Gamma^{\zeta_1, \zeta_2}_{s_1}(x_{12})$ as follows.
\begin{align*}
%%%%%%%%%%%%%% x_{12}
&\quad\Gamma^{\zeta_1, \zeta_2}_{s_1}\Psi_{\xi}^{s_2, 1}(x_{12})
=\Gamma^{\zeta_1, \zeta_2}_{s_1}(\xi x_{12}^{2s_2+1})\\
&=\frac{\zeta_1^{2s_2+1}\xi}{2^{2s_2+1}}
     \left[x_{11}^{2s_1}(x_{11}-x_{22})-x_{12}^{2s_1}(\zeta_2x_{12}-\zeta_2^{-1}x_{21})\right]^{2s_2+1}\\
&=\frac{\zeta_1^{2s_2+1}\xi}{2^{2s_2+1}}\left[x_{11}^{2s_1(2s_2+1)}(x_{11}-x_{22})^{2s_2+1}
     -x_{12}^{2s_1(2s_2+1)}(\zeta_2 x_{12}-\zeta_2^{-1}x_{21})^{2s_2+1}\right]\\
&=\frac{\zeta_1^{2s_2+1}\xi}2\left[x_{11}^{2s_1(2s_2+1)+2s_2}(x_{11}-x_{22})
     -\zeta_2^{2s_2}x_{12}^{2s_1(2s_2+1)+2s_2}(\zeta_2 x_{12}-\zeta_2^{-1}x_{21})\right]\\
&=\Gamma^{\zeta_1^{2s_2+1}\xi,\, \zeta_2^{2s_2+1}}_{s_1(2s_2+1)+s_2}(x_{12}),
\\
%\end{align*}
%\begin{align*}
&\quad\Psi_{\xi}^{s_2, 1}\Gamma^{\zeta_1, \zeta_2}_{s_1}(x_{12})\\
&=\frac{\zeta_1}2\Psi_{\xi}^{s_2, 1}\left(x_{11}^{2s_1}(x_{11}-x_{22})
     -x_{12}^{2s_1}(\zeta_2x_{12}-\zeta_2^{-1}x_{21})\right)\\
&=\frac{\zeta_1}2\left[x_{11}^{2s_1(2s_2+1)+2s_2}(x_{11}-x_{22})
    -\xi^{2s_1}x_{12}^{2s_1(2s_2+1)+2s_2}(\zeta_2\xi x_{12}-(\zeta_2\xi)^{-1}x_{21})\right]\\
&=\Gamma^{\zeta_1,\, \xi^{2s_1+1} \zeta_2}_{s_1(2s_2+1)+s_2}(x_{12}).
\end{align*}
\end{proof}

\begin{lemma}
Let $\Phi_{\xi}^{s_2, 1}, \Gamma^{\zeta_1, \zeta_2}_{s_1}$ be two automorphisms of $A_{N2}^{\mu+}$
as defined in Theorem \ref{MainProof}, then
\begin{align}\label{AutRelationsN2_3}
\Phi_{\xi}^{s_2, 1}\Gamma^{\zeta_1, \zeta_2}_{s_1}
=\Gamma^{-\zeta_1,\,\xi^{2s_1-1}\zeta_2^{-1} }_{s_1(2s_2+1)+s_2}, \quad
\Gamma^{\zeta_1, \zeta_2}_{s_1}\Phi_{\xi}^{s_2, 1}
=\Gamma^{\xi \zeta_1^{-(2s_2+1)},\, -\zeta_2^{2s_2+1}}_ {s_1(2s_2+1)+s_2}.
\end{align}
\end{lemma}
\begin{proof}
It is a direct computation. Here we only calculate
$\Phi_{\xi}^{s_2, 1}\Gamma^{\zeta_1, \zeta_2}_{s_1}(x_{12})$ and
$\Gamma^{\zeta_1, \zeta_2}_{s_1}\Phi_{\xi}^{s_2, 1}(x_{12})$ as follows.
\begin{align*}
&\quad \Phi_{\xi}^{s_2, 1}\Gamma^{\zeta_1, \zeta_2}_{s_1}(x_{12})\\
&=\frac{\zeta_1}2\Phi_{\xi}^{s_2, 1}
     \left(x_{11}^{2s_1}(x_{11}-x_{22})-x_{12}^{2s_1}(\zeta_2 x_{12}-\zeta_2^{-1}x_{21})\right)\\
&=\frac{\zeta_1}2\left[(x_{11}^{2s_2}x_{22})^{2s_1}x_{11}^{2s_2}(x_{22}-x_{11})
     -(\xi x_{12}^{2s_2}x_{21})^{2s_1}x_{12}^{2s_2}(\zeta_2\xi x_{21}-\zeta_2^{-1}\xi^{-1}x_{12})\right]\\
&=\frac{\zeta_1}2\left[x_{11}^{2s_1(2s_2+1)+2s_2}(x_{22}-x_{11})
     -\xi^{2s_1}x_{12}^{2s_1(2s_2+1)+2s_2}
     (\zeta_2\xi x_{21}-\zeta_2^{-1}\xi^{-1}x_{12})\right]\\
&=\Gamma^{-\zeta_1,\,\xi^{2s_1-1}\zeta_2^{-1} }_{s_1(2s_2+1)+s_2}(x_{12}),\\
%%%%%%%%%%%%%
&\quad \Gamma^{\zeta_1, \zeta_2}_{s_1}\Phi_{\xi}^{s_2, 1}(x_{12})
=\Gamma^{\zeta_1, \zeta_2}_{s_1}(\xi x_{12}^{2s_2}x_{21})
=\xi \Gamma^{\zeta_1, \zeta_2}_{s_1}(x_{21}^{2s_2+1}) \\
&=\xi \left(\frac{\zeta_1^{-1}}2\right)^{2s_2+1}
   \left[x_{11}^{2s_1}(x_{11}-x_{22})+x_{12}^{2s_1}(\zeta_2x_{12}-\zeta_2^{-1}x_{21})\right]^{2s_2+1} \\
&=\xi \left(\frac{\zeta_1^{-1}}2\right)^{2s_2+1}
   \left[x_{11}^{2s_1(2s_2+1)}(x_{11}-x_{22})^{2s_2+1}
   +x_{12}^{2s_1(2s_2+1)}(\zeta_2x_{12}-\zeta_2^{-1}x_{21})^{2s_2+1}\right] \\
&=\frac{\xi \zeta_1^{-(2s_2+1)}}2
       \left[x_{11}^{2s_1(2s_2+1)+2s_2}(x_{11}-x_{22})
   +\zeta_2^{2s_2}x_{12}^{2s_1(2s_2+1)+2s_2}(\zeta_2x_{12}-\zeta_2^{-1}x_{21})\right] \\
&=\Gamma^{\xi \zeta_1^{-(2s_2+1)},\, -\zeta_2^{2s_2+1}}_ {s_1(2s_2+1)+s_2}(x_{12}).
\end{align*}
\end{proof}

\begin{theorem}
\begin{enumerate}
\item If $(n, \lambda)\neq (2, 1)$ or $(n, \lambda, \mu, N)=(2, 1, -1, \text{even})$, then
\[
{\rm Aut}_{\rm Hopf}\left(A_{Nn}^{\mu\lambda}\right)
=\left\{\Psi_{\xi}^{s,t},  \Phi_{\xi}^{s,t}\right\}
\]
with relations given by \eqref{AutRelationsMain},
where $s\in\overline{1, N}$, $t\in\overline{1, n-1}$,  $t\equiv 1({\rm mod }\,\, 2)$,
$(2s+t, N)=1=(n, t)$,
$\left\{\begin{array}{ll}
\xi^{2N}=1=\xi^4, &\text{$n$  even},\vspace{1mm}\\
\xi^2=1, &  \text{$n$ odd}.\\
\end{array}\right.$
\item If $(n, \lambda)=(2, 1)$ and $(\mu, N)\neq (-1, \text{even})$, then
\[
{\rm Aut}_{\rm Hopf}\left(A_{Nn}^{\mu\lambda}\right)
=\left\{\Psi_{\xi}^{s,1},  \Phi_{\xi}^{s,1}, \Gamma^{\zeta_1, \zeta_2}_s\right\}
\]
with relations given by \eqref{AutRelationsMain}, \eqref{AutRelationsN2_1},
\eqref{AutRelationsN2_2} and \eqref{AutRelationsN2_3},
where $s\in\overline{1, N}$, $(2s+1, N)=1$, $\xi^{2N}=1=\xi^4$, 
%$\left\{\begin{array}{ll}
%\xi^{2N}=1=\xi^4, &\text{$n$ even},\vspace{1mm}\\
%\xi^2=1, &  \text{$n$ odd}, \\
%\end{array}\right.$
$\left\{\begin{array}{ll}
\zeta_1^4=\zeta_2^4=\mu=1, &\text{$N$ even,}\vspace{1mm}\\
\zeta_1^2=\zeta_2^2=\mu, &\text{$N$  odd. }\\
\end{array}\right.$
\end{enumerate}
\end{theorem}
\begin{proof}
It is a summary of Theorem \ref{MainProof} and Formulas \eqref{AutRelationsMain}, \eqref{AutRelationsN2_1},
\eqref{AutRelationsN2_2} and \eqref{AutRelationsN2_3}.
\end{proof}

\begin{example}
${\rm Aut}_{\rm Hopf}\left(A_{12}^{\mu-}\right)
=\left\{\Psi_{\pm \mu}^{1,1},  \Phi_{\pm \mu}^{1,1}\right\}
$ with
\begin{align*}
\Psi_{\mu}^{1,1}={\rm id}=\left(\Psi_{\pm \mu}^{1,1}\right)^2=\left(\Phi_{\pm \mu}^{1,1}\right)^2, \quad
\Psi_{-\mu}^{1,1}=\Phi_{-1}^{1,1}\Phi_{1}^{1,1}=\Phi_{1}^{1,1}\Phi_{-1}^{1,1}.
\end{align*}
So the automorphism group of $A_{12}^{\mu-}$ is isomorphic to the Klein four-group.
\end{example}

\begin{remark}
The Kac-Paljutkin algebra $H_8$ \cite{MR0208401} is isomorphic to $A_{12}^{\mu-}$ as Hopf algebras.
The automorphism group of $H_8$ was obtained by \cite{MR2879228} and was used to determine
isomorphic classes of Hopf algebras over $H_8$ \cite{Shi2019}.
\end{remark}

\begin{example}
${\rm Aut}_{\rm Hopf}\left(A_{12}^{\mu+}\right)
=\left\{\Psi_{\xi}^{1,1},  \Phi_{\xi}^{1,1},
\Gamma_1^{\zeta_1,\zeta_2}\,\,\big|\,\,\xi^2=\zeta_1^2=\zeta_2^2=1\right\}
$ with
\begin{align*}
\Gamma^{a_1, a_2}_{1}\Gamma^{b_1, b_2}_{1}
&=\left\{\begin{array}{ll}
     \Psi_{a_2b_1}^{4,\, 1}=\Psi_{\mu a_2b_1}^{1,\, 1}, &a_1b_2=1,\vspace{1mm}\\
     \Phi_{a_2b_1}^{4,\, 1}=\Phi_{\mu a_2b_1}^{1,\, 1}, &a_1b_2=-1,\\
     \end{array}\right. \\
\Gamma^{\zeta_1, \zeta_2}_{1}\Psi_{\xi}^{1, 1}
&=\Gamma^{\zeta_1\xi,\, \zeta_2}_{4}=\Gamma^{\zeta_1\xi,\, \mu\zeta_2}_1,\quad
\Psi_{\xi}^{1, 1}\Gamma^{\zeta_1, \zeta_2}_{1}
=\Gamma^{\zeta_1,\, \xi \zeta_2}_{4}=\Gamma^{\zeta_1,\, \mu\xi \zeta_2}_{1},\\
\Phi_{\xi}^{1, 1}\Gamma^{\zeta_1, \zeta_2}_{1}
&=\Gamma^{-\zeta_1,\,\xi\zeta_2}_{4}=\Gamma^{-\zeta_1,\,\mu\xi\zeta_2}_{1}, \quad
\Gamma^{\zeta_1, \zeta_2}_{1}\Phi_{\xi}^{1, 1}
=\Gamma^{\xi \zeta_1,\, -\zeta_2}_ {4}=\Gamma^{\xi \zeta_1,\, -\mu\zeta_2}_ {1}.
\end{align*}
%\begin{align*}
%\left(\Gamma_1^{\zeta_1,\zeta_2}\right)^2
%&= \left\{\begin{array}{ll}
%      \Psi_{\mu}^{1,1}={\rm id}, &\zeta_1\zeta_2=1,\vspace{1mm}\\
%     \Phi_{-\mu}^{1,1}, &\zeta_1\zeta_2=-1,\\
%     \end{array}\right.\\
%\Gamma^{\zeta_1,\zeta_2}_1 \Gamma^{\zeta_1,-\zeta_2}_1
%&=\left\{\begin{array}{ll}
%     \Phi_{\mu}^{1,1}, &\zeta_1\zeta_2=1, \vspace{1mm}\\
%     \Psi_{-\mu}^{1,1}, &\zeta_1\zeta_2=-1, \\
%     \end{array}\right.\\
%\Gamma^{\zeta_1,\zeta_2}_1 \Gamma^{-\zeta_1,\zeta_2}_1
%&=\left\{\begin{array}{ll}
%     \Psi_{-\mu}^{1,1}, &\zeta_1\zeta_2=1, \vspace{1mm}\\
%     \Phi_{\mu}^{1,1}, &\zeta_1\zeta_2=-1, \\
%     \end{array}\right.
%\end{align*}
So ${\rm Aut}_{\rm Hopf}\left(A_{12}^{\mu+}\right)
=\left\langle \Gamma_{1}^{1,-1}, \Gamma_1^{1,1}\,\big|\,
  \left(\Gamma_{1}^{1,-1}\right)^4={\rm id}=\left(\Gamma_1^{1,1}\right)^2,
  \Gamma_1^{1,1}\Gamma_{1}^{1,-1}\Gamma_1^{1,1}=\left(\Gamma_{1}^{1,-1}\right)^{-1}\right\rangle$.
That is to say, ${\rm Aut}_{\rm Hopf}\left(A_{12}^{\mu+}\right)$ is isomorphic
to the dihedral group $D_8$.
\end{example}

\begin{remark}
Let $H$ be a finite dimensional Hopf algebra and $H^*$ be the dual hopf algebra of $H$,
then ${\rm Aut}_{\rm Hopf}(H)\cong {\rm Aut}_{\rm Hopf}(H^*)$.
Let $D_{2n}=\langle a,b \mid a^n=b^2=1, bab=a^{-1}\rangle$ be the dihedral group of order $2n$, then
$A_{12}^{++}\cong (\k D_8)^*$ from \cite[Remark 3.4]{MR1800713}.
Let us check that ${\rm Aut}_{\rm Hopf}(A_{12}^{++})\cong {\rm Aut}_{\rm Hopf}(\k D_{8})$.
According to \cite{Rotmaler1977} and \cite[Page 124]{Li2020}, when $n\geq 3$ we have
\begin{align*}
%D_{2n}&=\langle a,b \mid a^n=b^2=1, bab=a^{-1}\rangle\\
{\rm Aut}(D_{2n})
&=\left\{\phi_{i,j}\mid \phi_{i,j}(a)=a^i, \phi_{i,j}(b)=ba^{j}, (i,n)=1, i, j\in\Bbb{Z}_n\right\}, \\
{\rm Aut}(D_{8})&=\left\{\phi_{i,j}\,\big|\, \phi_{i,j}(a)=a^i, \phi_{i,j}(b)=ba^{j}, i\in\{1,3\}, j\in\Bbb{Z}_4\right\}\\
&=\left\langle \phi_{1,1}, \phi_{3,1}\,\big|\, \phi_{1,1}^4={\rm id}=\phi_{3,1}^2,
     \phi_{3,1}\phi_{1,1}\phi_{3,1}=\phi_{1,1}^{-1}\right\rangle\\
&\cong D_8.
\end{align*}
So ${\rm Aut}_{\rm Hopf}\left(A_{12}^{\mu+}\right)\cong D_8\cong
{\rm Aut}(D_{8})\cong {\rm Aut}_{\rm Hopf}\left(\k D_8\right)$.
\end{remark}

%%\bibliographystyle{amsalpha}
%\bibliographystyle{../CommAlg2022}
%%\bibliographystyle{elsarticle-num}
%%\bibliographystyle{elsarticle-num-names}
%%\bibliographystyle{elsarticle-harv}
%\bibliography{../suzuki}
%%\bibliographystyle{plain}
%%\bibliography{/Users/rhinodrink/Papers/My_sdudy/论文学习/kaiti2012}
\end{document}